\documentclass[12pt,oneside]{amsart}
\usepackage{graphicx, amssymb,amscd, verbatim}
\usepackage{euscript}
\usepackage{color}

      \theoremstyle{plain}
      \newtheorem{theorem}{Theorem}[section]
      
      \newtheorem{lemma}[theorem]{Lemma}
      \newtheorem{corollary}[theorem]{Corollary}
      \newtheorem{proposition}[theorem]{Proposition}
      \newtheorem{remark}[theorem]{Remark}

\numberwithin{equation}{section}

      \makeatletter
      \def\@setcopyright{}
      \def\serieslogo@{}
      \makeatother

\def\h{\Phi}

\def\th{\tilde h}

\def\H{\mathcal H}
\def\P{\mathcal H}
\def\D{\mathcal D}
\def\w{\mathcal W}
\def\E{\mathcal E}

\def\M{\mathcal{M}}
\def\E{\mathcal{E}}

\def\p{\mathcal {P}}

\def\R{\mathbb R}
\def\Q{\mathbb Q}
\def\C{\mathbb C}
\def\Z{\mathbb Z}
\def\N{\mathbb N}
\def\T{\mathbb T}
\def\dist{\text{dist}}

\def\Id{\text{Id}}

\def\e{\varepsilon}
\def\a{c}
\def\Ci{C^\infty}
\def\Cr{C^r}

\def\f{\phi}

\begin{document}

\date{\today}
\author{Andrey Gogolev$^\ast$, Boris Kalinin$^{\ast\ast}$, Victoria Sadovskaya$^{\ast\ast\ast}$}

 \address{Department of Mathematics, The Ohio State University,  Columbus, OH 43210, USA}
\email{gogolyev.1@osu.edu}

\address{Department of Mathematics, The Pennsylvania State University, 
University Park, PA 16802, USA}
\email{kalinin@psu.edu, sadovskaya@psu.edu}

\title[Center foliation rigidity]{Center foliation rigidity for partially hyperbolic toral diffeomorphisms}

\thanks{$^\ast$ Supported in part by NSF grant DMS-1823150}
\thanks{$^{\ast\ast}$  Supported in part by Simons Foundation grant 426243}
\thanks{$^{\ast\ast\ast}$ Supported in part by NSF grant DMS-1764216}

\begin{abstract}
We study perturbations of a partially hyperbolic toral automorphism $L$ which is diagonalizable 
over $\C$ and has a dense center foliation. For a small perturbation of $L$ with a smooth center 
foliation we establish existence of a smooth leaf conjugacy to $L$. We also show that if a 
small perturbation of an ergodic irreducible $L$ has smooth center foliation and is 
bi-H\"older conjugate to $L$, then the conjugacy is smooth. As a corollary, we show that for 
any symplectic perturbation of such an $L$ any bi-H\"older conjugacy must be smooth.
 For a totally irreducible $L$ with two-dimensional center, we establish a number of equivalent 
 conditions on  the perturbation that ensure smooth conjugacy to $L$.
 
\end{abstract}

\maketitle


\section{Introduction and statements of results}

Partially hyperbolic  ergodic toral automorphisms, which are sometimes called quasi-hyperbolic, 
form an important class of algebraic partially hyperbolic systems. They have been
extensively studied and shown to have strong stochastic and other properties,
often similar to those of hyperbolic systems: Bernoulli property \cite{Kz}, 
uniqueness of the measure of maximal entropy \cite{B}, exponential mixing \cite{L},
density of periodic measures \cite{M} and their asymptotic equidistribution \cite{L}, and
cohomological properties similar to Liv\v{s}ic periodic point theorem and measurable Liv\v{s}ic theorem \cite{V}.

Perturbations of partially hyperbolic  ergodic toral automorphisms give a natural class
of partially hyperbolic systems. In contrast to linear models,  the properties
of such perturbations are much less understood. Some of the difficulties presented by these 
nonlinear systems are due to multidimensional non-compact center leaves. 
 For totally irreducible ergodic toral automorphisms
with two-dimensional center foliation, stable ergodicity was established by Rodriguez Hertz in \cite{RH}. 
Further properties for this case, including the stable Bernoulli property for symplectic perturbations, 
were obtained by Avila and Viana in \cite{AV}.  

In this paper we study rigidity properties for perturbations of partially hyperbolic toral automorphisms
related to the smoothness of their center foliation. In particular, we obtain smoothness of the
leaf conjugacy to the linear system, and smoothness of the conjugacy when one exists. Our main results
hold for systems with dense center foliation of any dimension, but have no analogs in the hyperbolic case. 
Further results are then deduced for systems with two-dimensional center foliation using \cite{RH,AV}. 

\vskip.1cm
We consider a linear map  $L\in SL(d,\Z)$ and use the same notation for the corresponding
 toral automorphism $L:\T^d\to\T^d$. The map $L$ is called {\em irreducible}\, if it has 
no rational invariant subspaces, or equivalently if its characteristic polynomial 
is irreducible over $\Q$. The automorphism $L$ is ergodic with respect to the Lebesgue measure 
if and only if no root of unity is its eigenvalue. 
 We define the stable, unstable, and center subspaces $E^s$, $E^u$, and $E^c$ for $L$
 as those corresponding to eigenvalues of modulus less than 1, greater than 1, 
 and equal to 1, respectively.  
 We denote  by $W^s$, $W^u$, and $W^c$ the corresponding 
 linear foliations. An irreducible ergodic automorphism $L$ is always {\em partially hyperbolic},
 that is, it has non-trivial $E^s$ and  $E^u$. 
 We will consider  partially hyperbolic automorphisms  $L$ with non-trivial center $E^c$.

We consider a $C^\infty$ diffeomorphism   $f$ which is 
$C^{1}$ close to $L$. Such $f$ is {\em partially hyperbolic}, more precisely,
there exist a nontrivial $Df$-invariant splitting $\E^s\oplus \E^c \oplus \E^u$ of the tangent 
bundle of $\T^d$, a continuous Riemannian metric on $\T^d$, and constants 
$\nu<1,\,$ $\hat\nu >1,\,$  $\gamma,$ $\hat\gamma\,$ such that for any $x \in \M$ 
and any unit vectors  
$\,v^s\in \E^{s}(x)$, $\,v^c\in \E^{c}(x)$, and $\,v^u\in \E^{u}(x)$,
$$
\|D_xf(v^s)\| < \nu <\gamma <\|D_xf(v^c)\| < \hat\gamma <
\hat\nu <\|D_xf(v^u)\|.
$$
 The sub-bundles $\E^{s}$, $\E^{u}$, and $\E^{c}$ are called {\em stable, unstable, and center.}
The  stable and unstable sub-bundles are tangent to the stable and unstable
foliations $\w^{s}$ and $\w^{u}$, respectively. The leaves of these  foliations are
$C^\infty$.  
By structural stability of partially hyperbolic systems \cite{HPS}, $f$ is {\em dynamically 
coherent}, that is, the bundles $\E^{c}$, $\E^{cu}=\E^u\oplus \E^c$, and $\E^{cs}=\E^s\oplus \E^c$ 
are tangent to foliations $\w^{c}$, $\w^{cu}$, and $\w^{cs}$ with $C^{r}$ leaves, where $r>1$ is 
determined by expansion/contraction in $\E^c$ relative to the rates for $\E^u$ and $\E^s$. 
Moreover, $f$ is leaf conjugate to $L$ 
by a bi-H\"older homeomorphism $h$ close to the identity. {\em A leaf conjugacy} is a homeomorphism 
$h\colon\T^d\to\T^d$ mapping the  leaves of $\w^c$ homeomorphically to the  leaves of $W^c$ such that  
$$
   h(f(\w^c(x)))=W^c(L(h(x))) \quad\text{for every }x\in \T^d.
$$

Now we formulate our main results. First we establish existence of a smooth leaf conjugacy 
for a perturbation with a smooth center foliation.

 \begin{theorem} [Smooth leaf conjugacy] \label{localL} 
 Let $L:\T^d\to\T^d$ be a partially hyperbolic automorphism which is diagonalizable 
 over $\C$ and has dense center foliation $W^c$. 
 Let $f:\T^d\to\T^d$ be a $C^{\infty}$ diffeomorphism which is $C^1$ close to $L$. 
 If $\w^c$ is a $C^{\infty}$ foliation, then $f$ is $C^\infty$ leaf-conjugate to $L$.
 \end{theorem}

We note that the theorem applies, in particular, to all irreducible partially hyperbolic automorphisms.
Also, if the theorem applies to automorphisms $L_1$ and $L_2$, then it applies to $L_1 \times L_2$ 
and $L_1 \times \Id_{\T^k}$ as well.

\begin{remark} 
Theorem \ref{localL} has a finite regularity version: if $\w^c$ is a $C^{r}$ foliation
 with $r>r(L)$ from \eqref{r(L)} below, then $f$ is $C^{q}$ leaf-conjugate to $L$
 where $q=r$ if $r\notin \N$ and $q=r-\e$ for any $\e>0$  if $r \in \N$.
 This can be obtained by the same argument using $C^{r}$ normal form coordinates 
 and Journ\'e's lemma \cite{J}. 
 \end{remark}

 \vskip.1cm
Next we consider the case when $f$ is bi-H\"older conjugate to $L$. That is, we assume that there exists a H\"older continuous conjugacy $h$ with a H\"older continuous inverse. We obtain $C^\infty$ smoothness 
of this conjugacy if $\w^c$ has sufficient regularity defined as follows.
Let $1<\rho^u_{\min}\le \rho^u_{\max}$ be the smallest and largest moduli of unstable eigenvalues of $L$,
and let $0<\rho^s_{\min}\le \rho^s_{\max}<1$ be the smallest and largest moduli of its stable eigenvalues. 
We set  
\begin{equation}\label{r(L)}
\begin{aligned}
& r^u(L)= (\log \rho^u_{\max}) / (\log {\rho^u_{\min}})\ge 1, \\
& r^s(L)= (\log \rho^s_{\min}) / (\log {\rho^s_{\max}})\ge 1,\\
 & r(L)=\max \,\{r^u(L), \,r^s(L)\}.
 \end{aligned}
\end{equation}

  \begin{theorem}  [Smoothness of bi-H\"older conjugacy] \label{localG} 
  Let $L:\T^d\to\T^d$ be an irreducible ergodic automorphism  and let $r>r(L)$.  Let $f:\T^d\to\T^d$ be a volume-preserving 
$C^{\infty}$ diffeomorphism that is sufficiently $C^1$ close to $L$. If $f$ has $C^r$ center foliation and is conjugate to $L$ by a  bi-H\"older homeomorphism $h$,
then $h$ is $C^\infty$.
 \end{theorem}
 
 \begin{remark} 
If $E^u$ and $E^s$ are one-dimensional it suffices to take $r=1$ rather than $r>r(L)=1$, that is,
to assume that the center foliation is $C^1$. Indeed, for one-dimensional leaves the analog of
the centralizer part of Theorems \ref{NFext} and \ref{MainNF} was proved in \cite{KL} in $C^1$ regularity.
 \end{remark}

For a symplectic perturbation $f$ we obtain the following corollary.
 
 \begin{corollary} \label{symplecticG}
 Let $L:\T^d\to\T^d$ be a symplectic irreducible ergodic automorphism and let $f:\T^d\to\T^d$ be a $C^{\infty}$ symplectic diffeomorphism 
 which is $C^{1}$-close to $L$. If $f$ is bi-H\"older conjugate to $L$, then $f$ is $C^{\infty}$  conjugate to $L$.
  \end{corollary} 

This result is a rare example of rigidity in smooth dynamics, in the sense of  ``weak equivalence
 implies strong equivalence", that holds for a single system rather than an action of higher rank group.
 It relies on coexistence of hyperbolic and elliptic behavior in one system, and thus is also a rare example
 of a result for partially hyperbolic systems that does not cover hyperbolic systems as a  particular case. 

\begin{remark} \label{localGdense}
Theorem \ref{localG} and Corollary \ref{symplecticG} hold 
more generally for any automorphism $L$ which is partially hyperbolic, ergodic, diagonalizable 
over $\C$ and has a dense center foliation. This class of automorphisms includes products of irreducible ergodic 
automorphisms. The proof is essentially the same utilizing Proposition~\ref{normal holonomy L}.
See Remark~\ref{normal holonomy G}.
 \end{remark}

\vskip.1cm

Now we consider the case of $L$ with two-dimensional center.
We call a toral automorphism $L$ {\em totally irreducible} if $L^n$ is irreducible for every 
$n \in \N$. Such an $L$ is always ergodic.  For a totally irreducible automorphism $L$ with 
exactly two eigenvalues of absolute value one, that is $\dim E^c=2$, Rodriguez Hertz proved
in  \cite{RH} that it is stably ergodic, more precisely, any sufficiently $C^{22}$-small 
volume-preserving  perturbation of $L$ is also ergodic.  For such $L$ we use some results from 
\cite{RH}  and \cite{AV} to obtain further corollaries of Theorem \ref{localG}.
\vskip.1cm

We recall definitions accessibility and Lyapunov exponents before stating further results.
A partially hyperbolic  diffeomorphism $f$ of $\T^d$  is called {\em accessible}  if any two points 
 in $\T^d$ can be connected by an $su$-path, that is, by a concatenation 
 of finitely many subpaths each lying  in a single  leaf of $\w^s$ or  $\w^u$. 
 
 Let $\mu$ be an ergodic $f$-invariant measure.
Then by Oseledets Multiplicative Ergodic Theorem \cite{O} \,
there exist numbers $\lambda_1 < \dots < \lambda_{m}$, called the {\em Lyapunov exponents}
of $f$ with respect to $\mu$, an $f$-invariant 
set $\Lambda$ with $\mu (\Lambda)=1$, and a $Df$-invariant Lyapunov splitting 
$
\R^d=T_x\T^d  =\E^{1}_x\oplus\dots\oplus \E^{m}_x$  for $\,x\in \Lambda$ such that 
$$ 
\underset{n\to{\pm \infty}}{\lim} n^{-1} \log\| D_xf^n (v) \|=  \lambda_i\, 
\,\text{  for any }i=1,\dots ,m \,\text{ and any }\,0 \not= v\in \E^{i}_x.
$$
Clearly, the Lyapunov splitting refines the partially hyperbolic one.

\vskip.1cm
 In the next theorem and corollary we set $N=5$ if $d>4$ and $N=22$ if $d=4$.

 \begin{theorem}  [Rigidity for two-dimensional center] \label{local} 
 Let $L:\T^d\to\T^d$ be a totally irreducible automorphism with exactly two eigenvalues 
 of absolute value one and let $r>r(L)$.  Let $f:\T^d\to\T^d$ be a volume-preserving 
$C^{\infty}$ diffeomorphism which is sufficiently $C^N$ close to $L$ and has $C^r$ center foliation.
 Then any of the following equivalent conditions implies that  $f$ is $C^{\infty}$ conjugate to $L$.
\vskip.1cm
\begin{itemize}
\item[(1)] Lyapunov exponents of $f$ with respect to the volume on $\E^c$ are all 0;

\item[(2)] Lyapunov exponents of $f$ with respect to the volume on $\E^c$ are equal;

\item[(3)] $f$ is not accessible;

\item[(4)] $f$ is topologically conjugate to $L$;

\item[(5)] $\w^s$ and $\w^u$ are jointly integrable, that is, there exists a continuous foliation
of dimension $\dim \w^s + \dim \w^u$ sub-foliated by $\w^s$ and $\w^u$.
\end{itemize}
\end{theorem}

 \begin{corollary} \label{symplectic}
 Let $L$ be as in Theorem \ref{local} and symplectic, and let $f:\T^d\to\T^d$ be  
 a $C^{\infty}$ symplectic diffeomorphism  which is sufficiently $C^{N}$-close to $L$. 
 Then any of the following equivalent conditions implies that  $f$ is $C^{\infty}$ conjugate to $L$.
\vskip.1cm
\begin{itemize}
\item[(0)] $f$ has at least one zero Lyapunov exponent with respect to the volume;

\item[(1-5)] as in Theorem \ref{local};

\item[(6)] $\E^s \oplus \E^u$ is $C^1$.
\end{itemize}
 \end{corollary} 
 
 \begin{remark}  \label{dichotomy}
 Thus for perturbations as in Corollary \ref{symplectic} we have a dichotomy:
 either  $f$ is non-uniformly hyperbolic or $f$ is smoothly conjugate to $L$.
 \end{remark}

\noindent This  strengthens the earlier result in the same setting 
\cite[Theorem I]{AV} which showed that either $f$ is non-uniformly hyperbolic
or $f$ is conjugate to $L$ via a volume-preserving homeomorphism. Smoothness 
of the conjugacy was only known for the case of $\T^4$~\cite{AV}.
 
\vskip.1cm

 Our results are somewhat similar to
some of the recent rigidity results for partially hyperbolic systems related to absolute
continuity of the center foliation 
\cite{AVW1, AVW2, SY, DX}. 
For example, our Theorem~\ref{localL} can be compared to a theorem of Avila-Viana-Wilkinson~\cite{AVW2} on $\T^3$. Namely, they consider volume preserving perturbations $f$ of a partially hyperbolic automorphism $L(x,y)=(Ax, y)$ of the 3-torus $\T^3$.  Then, by applying the invariance principle~\cite{AV}, they show that the center foliation is absolutely continuous if and only if it is smooth. Consequently, $f$ is smoothly conjugate to a diffeomorphism of the form $(x,y)\mapsto (g(x), y+\varphi(x))$.
This result was generalized to the case of higher dimensional compact center foliation by Damjanovic and Xu~\cite[Theorem 6]{DX}.

We note that papers~\cite{AVW1, AVW2, SY, DX} consider diffeomorphisms whose 
 center foliation either has compact leaves or  comes from the orbit foliation of a hyperbolic flow. 
Further, they also strongly rely on one-dimensionality of stable and unstable foliations (or a replacement assumption such as quasi-conformality or splitting into one-dimensional subbundles). 
 In contrast our methods treat all dimensions in a uniform way and primarily rely on denseness of center leaves and the theory of normal forms~\cite{GK, G, KS17, K19}.

\vskip.2cm

\noindent {\bf Structure of the paper.} In Section \ref{normal} we summarize results on normal forms that play an important part in our arguments.
Then we prove Theorem \ref{localG} in Section~\ref{PlocalG}. The existence of the conjugacy
in this case allows us to present one of the main arguments, smoothness along stable/unstable foliations
via normal forms and holonomies, in a simplified form. We deduce Corollary \ref{symplecticG}, Theorem \ref{local}, and 
Corollary \ref{symplectic} in Section \ref{Pcor}. In section Section \ref{PlocalL} 
we prove Theorem \ref{localL}, giving modifications needed to carry out the normal forms and holonomies arguments in the case of leaf conjugacy.

\vskip.2cm

\noindent {\bf Acknowledgments.} We would like to thank Federico Rodriguez Hertz and  Ralf Spatzier
for useful discussions.


 \section{Normal forms for contractions} \label{normal}

In this section we give preliminaries on non-stationary normal forms for contractions.
To make the presentation less technical, we formulate the results only for perturbations of linear maps.
This is sufficient for our purposes.

Let $f$ be a homeomorphism of a compact connected manifold (or a compact metric space) $\M$.
Let $\E=\M\times \R^k$ be a vector bundle and let $U\subset \E $ be a neighborhood of the zero section.
We will consider a $\Cr$ extension $F$ of $f$, that is, a map $F : U\to \E$ that projects to $f$, 
preserves the zero section, and
such that the corresponding fiber maps $F_x: U_x \to \E_{f(x)}$ are $\Cr$ and depend continuously
on $x$ in $\Cr$ topology. We will assume that the derivative of $F$ at the zero section is sufficiently 
$C^0$ close on $\M$ to a constant linear contraction, that is,  $D_0 F_x$ is close  uniformly in $x$
to a fixed linear map $A \in GL(k,\R)$  with $\| A\| <1$. 

For any such matrix $A$ there exists a finite dimensional Lie group $\p_A$ with respect 
to composition which consists of certain polynomial maps $P : \R^k \to \R^k$ with $P(0)=0$ 
and invertible derivative at $0$. The elements of $\p_A$ are so called {\em sub-resonance 
generated polynomials}.   This group is determined by the (ratios of) absolute values 
$\chi_1  <\dots < \chi _\ell<0$ of  eigenvalues of $A$ and by the corresponding invariant 
subspaces. The degrees of these polynomials are bounded above by $d(A) =\chi_1 / \chi_\ell$, 
which yields that this group is finite dimensional. 
A precise definition of $\p_A$ can be found in \cite{GK,G}, but it does not play a role in this paper.


The following theorem was established in \cite{GK,G} for $r \in \N \cup \{\infty\}$, 
 in \cite{KS17} for any $r$ in nonuniformly hyperbolic setting,
 and in \cite{K19} for this setting.
 
\begin{theorem}[Normal forms for contracting extensions]\label{NFext}$\;$
Let $A \in GL(k,\R)$ with $\| A\| <1$, let $\e>0$ and $r\in [d(A)+\e, \infty]$.
Let $F:U\to \E$ be a $\Cr$ extension of $f$ whose derivative at the zero section is 
sufficiently $C^0$-close to $A$. 

Then there exist  a neighborhood $V$ of the zero section   
and a  family  $\{ \h_x\}_{x\in \M}$ of $\Cr$ diffeomorphisms   
 $\h_x : V_{x} \to \E_x$,  satisfying $\h_x(0)=0$ and $D_0 \h_x =\Id \,$ and depending
 continuously on $x$ in the $\Cr$ topology,
 which conjugate $F$ to a  polynomial extension $P$, i.e., for all $ x\in \M,$
 \begin{equation}\label{NF}
\h_{f(x)} \circ F_x =P_x \circ \h_x, \; \text{ where }  \;  P_x\in \p_{A}.
\end{equation}

Moreover,  let  $g:\M\to \M$ be a homeomorphism commuting with $f$
and let $G:U\to \E$ be a $C^{d(A)+\e}$ extension of $g$ 
preserving the zero section and commuting with $F$. Then  for all $x\in \M$,
 \begin{equation}\label{cent} 
 \h_{g(x)} \circ G_x \circ \h_x^{-1} \in \p_A.
\end{equation}
\end{theorem}

\begin{remark} [Global version] \label{NFglob}
Suppose that $F :\E\to \E$ is a globally defined extension which satisfies the assumptions of 
Theorem \ref{NFext} and either contracts fibers
or, more generally, satisfies the property that for any compact set $K \subset \E$ and any 
neighborhood $V$ of the zero section we have $F^n(K) \subset V$ for all sufficiently large $n$.
Then the family $\{ \h_x\}_{x\in \M}$ can be uniquely extended   ``by invariance"
$\h_x  = (P_x^n)^{-1} \circ  \h_{f ^n (x)} \circ F_x^n$ to the family of global $\Cr$ 
diffeomorphisms $\h_x :  \E_{x} \to \E_{x}$ satisfying \eqref{NF}.
Moreover, if $G$ is another extension which commutes with $F$, then it satisfies \eqref{cent} globally.
 \end{remark}

\vskip.2cm

These results can be applied in the context of foliations as follows.
Let $f$ be a diffeomorphism of a  compact connected manifold  $\M$, and  let $\w$ 
be an $f$-invariant continuous foliation of $\M$ with uniformly $\Ci$ leaves. 
The latter means that all leaves are $\Ci$ submanifolds and all their derivatives 
are continuous on $\M$. 

Suppose that $f$ contracts $\w$ and that the derivative $Df|_{T\w}$, as a linear extension 
on $\E=T\w$, is close to a constant $A$. Restricting $f$ to the leaves of $\w$ and identifying 
locally $\w_x=\w(x)$ with $T_x \w$, we obtain a corresponding non-linear extension $F$ as in 
Theorem \ref{NFext} and hence a family $\{ \h_x\}_{x\in \M}$ of local normal form coordinates,
Then, as in Remark \ref{NFglob}, they can be extended to global 
diffeomorphisms $\h_x :  \w_x \to \E_{x}$ satisfying \eqref{NF}.
The important new statements in this setting describing dependence along the leaves, parts 
(2) and (3) in the next theorem, were established in \cite{KS16}.

\begin{theorem}[Normal forms for contracting foliations, \cite{KS16}]\label{MainNF} 
Let $f$ be a $\Ci$ diffeomorphism of a smooth compact connected manifold  $\M$,
and let $\w$ be an $f$-invariant topological foliation of $\M$ with uniformly $\Ci$ leaves. 
Suppose that $\w$ is contracted by $f$, 
and that the linear extension $Df |_{T\w}$ is close to a constant $A$ as in Theorem \ref{NFext}.
Then there exists a family $\{ \h_x \} _{x\in \M}$ of $C^\infty$ diffeomorphisms  
$\,\h_x: \w_x \to \E_x =T_x\w$ such that  for each $x \in \M$,
 $$
 P_x =\h_{f(x)} \circ f \circ \h_x ^{-1}:\E_{x} \to \E_{f{(x)}}\text{ is in  }\p_A. 
 $$
The family $\{ \h_x \} _{x\in \M}$ has the following properties:  

\begin{enumerate} 
\item $\h_x(x)=0$ and $D_x\h_x $ is 
the identity map  for each $x \in \M$; 

\item $\h_x$ depends continuously on $x \in \M$ in $C^\infty$ topology
 and smoothly on $x$ along the leaves of $\w$; 
 
 \item For any  $x \in \M$ and $y \in \w_x$, the map $\h_y \circ \h_x^{-1} : \E_x \to \E_y$ 
is a composition of a sub-resonance generated polynomial in $\p_A$ with a translation;
 
\item If $g$ is a homeomorphism of $\M$ which commutes with $f$,  preserves $\w$, 
 and is $C^{d(A)+\e}$  along  the leaves of $\w$, then 
 for each $x \in \M$  
  $$
  Q_x=\h_{f(x)} \circ g \circ \h_x ^{-1}:\E_{x} \to \E_{g{(x)}}  \text{ is in  }\p_A.
  $$
  \end{enumerate}
\end{theorem}

 Another way to interpret (3) is to view $\h_x$ as a coordinate chart on 
$\w_x$, identifying it with $\E_x$, and in particular  identifying $\E_y=T_y \w_x$ with 
$T_{\h_x(y)}\E_x$ by $D_y\h_x $.  In this coordinate chart,  
(3) yields that all transition maps $\h_y \circ \h_x^{-1}$ for $y\in \w_x$ are  in the group 
generated by the translations of $\E_x$ and the sub-resonance generated polynomials,
which is isomorphic to the Lie group $\bar \p_A$ generated by $\p _A$ and the
 translations of $\R^k$.  Clearly, this group is also finite dimensional.


\section{Proof of Theorem \ref{localG}} \label{PlocalG}

By standard considerations we may assume that $h$ is homotopic to $id_{\T^d}$. 
Indeed, the induced linear map $h_*\colon \T^d\to\T^d$ is in the centralizer of $L$. 
Hence, by compositing with $h_*^{-1}$, we may assume that $h_*=id$, 
i.e., $h$ is homotopic to the identity map.  Note that $h$ does not have to be $C^0$ close to identity.

\subsection{Outline of the proof}
  We denote the stable, unstable, and center sub-bundles for $L$
 by $E^s$, $E^u$, $E^c$, and the ones for $f$ by $\E^s$, $\E^u$, $\E^c$.
 Similarly, we use $W$ and $\w$ for the corresponding foliations for $L$ 
 and  for $f$. Lemma \ref{hmaps} below  
 shows that the conjugacy $h$ respects the foliations,
so essentially we study its smoothness by restricting it to $\w^s$, $\w^u$, and $\w^c$.
The first part of the proof, Section \ref{hsmoothWs},  is showing smoothness along 
the stable and unstable foliations using normal forms and center holonomies.

The second  part of the proof is to establish uniform smoothness of $h$ along the center foliation. We first do it for the stable and unstable components
of $h$ in Section~\ref{Gsmooth} and then global smoothness of the stable and unstable components follows by the standard application of the Journ\'e's Lemma~\cite{J}.
Finally, wee use a different argument to establish global smoothness of the center component in Section~\ref{Csmooth}. 
  
The following lemma has a  rather standard proof and we include it for the sake of completeness.
 
 \begin{lemma} \label{hmaps} 
 Let $L$ be a partially hyperbolic toral automorphism and let $f$ be a dynamically coherent 
 partially hyperbolic toral diffeomorphism topologically conjugate to $L$ by a homeomorphism $h$.
 Then $h(\w^{*})=W^{*}$ for $\ast = s,u,c,cs,cu$. 
\end{lemma}

\begin{proof} We show that center unstable leaves for $f$ are mapped to those for $L$. 
With respect to a suitable metric, $L$ does not increase distances along $W^{cs}$, that is, 
$\dist (L^nx,L^ny)\le \dist (x,y)$ for any $y \in W^{cs}(x)$ and $n \in \N$. Then $h^{-1}(y)$ will remain close to $h^{-1}(x)$ under forward iterates of $f$. More precisely,  for any $\e >0$ there
is $\delta >0$ such that $\dist (f^n(h^{-1}(x)),f^n (h^{-1}(y)) )\le \e$ for any  $n \in \N$ and $y \in W^{cs}(x)$ with $\dist (x,y)< \delta$. If $\e$ is sufficiently small, this implies that $h^{-1}(y) \in \w^{cs}(h^{-1}(x))$, as otherwise they would separate exponentially along the unstable direction until reaching a ``moderate" distance $>\e$.
By connectedness, all points of $W^{cs}(x)$ must be mapped 
to the same center stable leaf of $f$, so we get $h^{-1}(W^{cs}(x)) \subset \w^{cs}(h^{-1}(x))$.
The equality follows from $h$ being a homeomorphism. Applying the Invariance of Domain Theorem
to $h^{-1}$ from a small ball in $W^{cs}(x)$ to 
$\w^{cs}(h^{-1}(x))$, we conclude that $h$ is a local homeomorphism between center stable leaves
of $f$ and $L$ on small balls of fixed size. By connectedness, all points of $\w^{cs}(h^{-1}(x))$ 
must come from the same center stable leaf of $L$.

Similarly, $\w^{cs}$ is mapped to $W^{cs}$, and it follows that  $\w^{c}$ is mapped to $W^{c}$
as the intersection of $\w^{cs}$ and $\w^{cu}$.

We also have $\w^{s}$ is mapped to $W^{s}$, and similarly for $\w^{u}$ and $W^{u}$. 
Indeed if $y \in \w^s(x)$ then $d(f^nx,f^ny) \to 0$ and hence $d(L^nh(x),L^nh(y)) \to 0$.
 It follows that $h(y) \in W^{s}(h(x))$ and so $h(\w^{s}(x)) \subset W^{s}(h(x))$. 
 The equality again follows since $h$ is a homeomorphism.
\end{proof}


 \subsection{Smoothness of the conjugacy along the stable leaves.} \label{hsmoothWs}
 
Since $f$ is a small perturbation of $L$, 
Theorem \ref{MainNF} applies and yields existence of the normal forms  on 
$\w^s$ and $\w^u$ corresponding to the groups of sub-resonance generated polynomials 
$\p_s=\p_{L|E^s}$ and $\p_u=\p_{L|E^u}$, respectively. 

Now we consider the holonomies $\H=\H^c$ of $\w^{c}$ inside $\w^{cs}$, that is, the maps 
 $$
 \H_{x,y} :\w^{s}(x) \to \w^{s}(y) \quad\text{given by}\quad
  \H_{x,y}(z)= \w^c(z)\cap  \w^{s}(y).
  $$
  The corresponding linear holonomies $H$ for $L$ are the translations along $W^c$:
  if $y\in W^c(x)$ then $H_{x,y}(z)=z+ (y-x)$.
Because $h$ maps stable leaves to stable leaves and center leaves to center leaves, 
the topological conjugacy $h$ intertwines the holonomies $\H$ and $H$, that is,
 $$
 H_{h(x),\,h(y)}=h\circ \H_{x,y} \circ h^{-1}.
 $$
It follows that $\H$ is globally defined on the leaves of $\w^s$ and is as smooth as $\w^c$.

We fix a pair of complex conjugate eigenvalues of $L$ of absolute value 1 and denote by $V$  
the corresponding invariant 2-dimensional subspace of $E^c$.  We will now show 
 that holonomies $\H$ which are conjugate to translations in $V$ preserve normal forms on  $\w^{s}$.

 \begin{proposition}  \label{normal holonomy} 
 For each $x\in \T^d$ and $y\in \w^c(x)$ with $h(x)-h(y) \in V$, the center holonomy 
 $\H_{x,y}:\w^{s}(x) \to \w^{s}(y)$  preserves normal forms on  $\w^{s}$, that is, 
 $ \h_y \circ  \H_{x,y} \circ  \h_x^{-1} :\E_{x}^s \to \E_{y}^s$ 
 is a sub-resonance generated polynomial in $\p_s$.
 \end{proposition}

\begin{remark} \label{normal holonomy G}
We will later extend this result in Proposition \ref{normal holonomy L}, which implies 
that {\em all} center holonomies preserve normal forms on $\w^{s}$. Using this we can replace
the assumption that $L$ is irreducibility by the  assumption that $L$ is
diagonalizable over $\C$ and has dense center foliation. This is  yield a somewhat
more general result given in Remark \ref{localGdense}
 \end{remark}

\begin{proof}
For any vector $v\in V$, the translation $H_v (x) =x+v$, $x\in\T^d$, is a globally defined map
whose restriction to any stable leaf is a center holonomy for $L$. While $L$ and $H_v$
do not commute, we have $L(H_v(x))=Lx+Rv$, where the restriction $R=L|_{V}$ is a 
linear map conjugate to the rotation by some angle $2\pi \theta$. We will denote by $R^{t}$ 
the corresponding conjugate of the rotation by the angle $2\pi \theta t$, for which $R^{1/\theta}=\Id$. 
Therefore, in order to apply Theorem \ref{MainNF}, we pass to the suspension flow
and use time-$1/\theta$ map.
The argument below is inspired by the one in \cite{FKSp11}.

\vskip.05cm

We consider the mapping tori
$$
M_f=\T^d\times [0,1]\,/ \,(x,1)\sim (f(x), 0)\,\,\,\,\,\textup{and}\,\,\,\, M_L=\T^d\times [0,1]\,/\, (x,1)\sim (L(x), 0)
$$
and the corresponding suspension flows $\{f^s\}$ and $\{L^s\}$ given by $(x,t)\mapsto (x, t+s)$. 
Then $h$ induces the conjugacy 
$\th\colon M_f\to M_L$  between the suspension flows $\th (x,t)=(h(x), t)$.
We also consider the map $T_v : \T^d\times \R \to \T^d\times \R$  given by
$$
T_v(x,t)=(x+R^{-t}v,\, t)
$$
and its projection $\tilde H_v: M_L\to M_L$ to $M_L$. The translation $H_v$  
embeds as $t=0$\, level of  the map $\tilde H_v$.
The projection $\tilde H_v$ is well-defined since 
$$
\begin{aligned} 
&T_v(x,1)=(x+R^{-1}v,\,1)\;\text{ is identified with }\\
&T_v(Lx,\,0)=(Lx+v,\,0)=(Lx+LR^{-1}v),\,0)=(L(x+R^{-1}v),\,0).
\end{aligned}
$$
We note that $T_v$ commutes with the map $(x,t)\mapsto (x, t+1/\theta)$ 
on $\T^d\times \R$. Indeed,  as $R^{1/\theta}=\Id$ we obtain
$$
T_v (x,\, t+1/\theta)= (x+R^{-t+1/\theta}v,\, t+1/\theta) =(x+R^{-t}v,\, t+1/\theta).
$$
It follows that  $\tilde H_v$ commutes with time-$1/\theta$ map $L^{1/\theta}$,
 as the projections to $M_L$. 
\vskip.1cm

Since $f$ is a small perturbation of $L$, Theorem \ref{MainNF} applies to the time-${1/\theta}$ map
of the suspension flow $\{f^s\}$ and yields existence of the normal form coordinates $\{\h_x\}$ on 
 its stable foliation $\tilde \w^{s}$ in $M_f$. In fact, the corresponding groups of sub-resonance 
generated polynomials $\p_s=\p_{L|E^s}$ are the same for all $t$. 

For any $v\in V$, we have that the map $g=\th^{-1}\circ \tilde H_v \circ \th: M_f \to M_f$ 
is a holonomy map of the lifted center foliation $\tilde \w^c$, and hence is $C^r$  along the leaves.
Since $\tilde H_v$ commutes with $L^{1/\theta}$ we obtain that $g$ commutes with $f^{1/\theta}$. 
Thus  part (4) of Theorem \ref{MainNF} applies and we conclude that  
$ \h_{g(x)} \circ  g \circ  \h_x^{-1} :\tilde \E_{x}^s \to \tilde \E_{g(x)}^s $ is a sub-resonance generated 
polynomial. In particular, this holds at the level $t=0$ of $M_f$ where $g$ coincides
with a holonomy map of $\w^c$ on $\T^d$. Moreover, any holonomy map $\H_{x,y}$ as
in the statement is given by $\th^{-1}\circ \tilde H_v \circ \th$ for some $v\in V$. Thus we 
conclude that $ \h_y \circ  \H_{x,y} \circ  \h_x^{-1} :\E_{x}^s \to \E_{y}^s$ 
 is a sub-resonance generated polynomial map.
\end{proof}

We fix arbitrary $x\in \T^d$ and $y\in \w^{s}(x)$. Since $L$ is irreducible and $V$ is $L$-invariant,
the linear foliation of planes parallel to $V$ has dense leaves in $\T^d$. Hence there exists a 
sequence of vectors $v_n\in V$ such that $h(x)+v_n$ converges to $h(y)$. Denoting 
$y_n=h^{-1}(h(x)+v_n)$ we obtain a sequence of points $y_n\in \w^c(x)$ converging to $y$
so that Proposition \ref{normal holonomy} applies to holonomies $\H_{x,y_n}: \w^s(x) \to \w^s(y_n)$.
The corresponding linear holonomies $H_{h(x),\,h(y_n)}=H_{v_n}$
for $L$ converge in $C^0$ to the translation $H_v$ in $W^s(h(x))$ by the vector $v=h(y)-h(x)$. Hence
the holonomies $\H_{x,y_n}$ converge in $C^0$ norm to some map $\H_{x,y}: \w^s(x) \to \w^s(y)$, which 
is the conjugate by $h$ of this linear translation.

By Proposition \ref{normal holonomy}, 
$\,\H_{x,y_n}$ is a sub-resonance generated polynomial map  $P_n$ in normal form coordinates, i.e.
 $$
 P_n =  \h_{y_n} \circ  \H_{x,y_n} \circ  \h_x^{-1} : \, \E_{x}^s \to \E_{y_n}^s.
 $$
Since the normal form coordinates $\h_y$ depend continuously on $y$,
the maps $P_n$ converge in $C^0$ to the map
$$
P=  \h_{y} \circ  \H_{x,y} \circ  \h_x^{-1} : \, \E_{x}^s \to \E_{y}^s,
$$
which is also a sub-resonance generated polynomial. 
Using (3) of Theorem \ref{MainNF} and identifying $\w^s(x)$ with $\E^s_x$ by the
$C^\infty$ coordinate map $\h_x$,
we see as in the remark after Theorem \ref{MainNF} that $P$ is in the Lie group $\bar \p _x$
generated by the translations of $\E_x^s$ and the sub-resonance generated polynomials,
which is isomorphic to the Lie group $\bar \p_A$ generated by $\p _A$ and the
 translations of $\R^k$. 
\vskip.1cm

Thus $h$ conjugates the action of $E^s=\R^k$ by translations of $W^s(h(x))$ 
with the corresponding continuous action of $\R^k$ by elements of the Lie group $\bar \p_x$
of $C^\infty$ polynomial diffeomorphisms of $\w^s(x)$. This conjugacy defines the
 continuous homomorphism 
 $$
 \eta_x : E^s \to  \bar \p_x \quad\text{given by}\quad 
 \eta_x (v)= h^{-1} \circ H_v \circ h.
 $$ 
It is a classical result that $\eta_x$ is automatically a $C^\infty$ homomorphism, see
for example \cite[Corollary 3.50]{Ha}.  Since  $\eta_x$ determines the conjugacy along the leaf by  
 $$
 h^{-1}(h(x)+v)=\eta_x (v)(x),
 $$
  we conclude that $h^{-1}$ is a $C^\infty$ diffeomorphism between $W^s(h(x))$ and 
  $\w^s(x)$, and hence $h$ is also $C^\infty$ along $\w^s(x)$. 
 
 Since the normal form coordinates $\h_x$, as well as holonomies and their limits, 
 depend continuously on $x$, the constructed continuous action on $\w^s(x)$ and 
 the corresponding homomorphism $\eta_x$ also depend continuously on $x$. 
 This implies that $\eta_x$ depend continuously on $x$ in $C^\infty$ topology,
 for example because it is determined by the corresponding linear homomorphism 
 of the Lie algebras. So we conclude that $h$ is uniformly $C^\infty$ along $\w^s$. 
 
 A similar argument shows that $h$ is uniformly $C^\infty$ along $\w^u$.

\begin{remark} 
The last part of the proof is similar to an argument pioneered by Katok and Spatzier in \cite{KtSp} and 
used in other papers on higher rank actions. In these arguments a continuous action 
by $C^\infty$ diffeomorphisms of $\w^s(x)$ is obtained. The smoothness of this action,
and hence of $h$, follows then from a more difficult result \cite[Section 5.1, Corollary]{MZ}.
This argument, however, does not immediately yield that $h$ is {\em uniformly} $C^\infty$ 
along $\w^s$. Our argument relies on the advanced results on normal forms from \cite{KS16}
to show that all maps $P$ are contained in a single Lie group $\bar \p_x$.
\end{remark}


\subsection{The conjugacy $h$ is volume-preserving} 
We denote the Lebesgue measure on $\T^d$ by $m$ and the $f$-invariant volume by $\mu$.
We will show that $h_*(\mu)=m$.

We denote the Lyapunov exponents of $f$ with respect to $\mu$ by $\lambda^f$,
and the Lyapunov exponents of $L$ by $\lambda^L$.
Since $Df|_{\E^u}$ is conjugate to $L|_{E^u}$ by the derivative of $h$ along $\w^u$, 
the Lyapunov exponents  of $f$ along $\E^u$ with respect to $\mu$ are equal to 
the unstable Lyapunov exponents  of  $L$.  Since $f$ and $L$ are topologically conjugate, 
they have the same topological entropy. 
Combining these observations with Pesin's formula for the metric entropy,
we obtain 
$$
\mathbf{h}_{top}(f)\ge \mathbf{h}_{\mu}(f)= \sum_{\lambda ^f >0} \lambda ^f \ge \sum_{\lambda ^f  \text{on $\E^u$}} \lambda ^f
=  \sum_{\lambda ^L >0} \lambda ^L = \mathbf{h}_{m}(L) = \mathbf{h}_{top}(L) =\mathbf{h}_{top}(f).
$$
Therefore, $\mathbf{h}_{top}(f)= \mathbf{h}_{\mu}(f)$, that is, $\mu$ is the measure of maximal entropy for $f$,
and so is the measure $h^{-1}_*(m)$ since $\mathbf{h}_{h^{-1}_*(m)}(f)=\mathbf{h}_{m}(L)$ by isomorphism.
We conclude that $h_*(\mu)=m$ by uniqueness of the measure of maximal entropy~\cite{B}.


\subsection  {Global smoothness of the stable and unstable components} \label{Gsmooth}
We already proved that $h$ is uniformly $C^\infty$ along $\w^{s}$ and $\w^{u}$. 
To show global smoothness of $h$ we now study its regularity along $\w^{c}$. 
For this we will decompose $h$ into stable, unstable, 
 and center components and consider them separately using their series representations.
In this section we will obtain uniform smoothness along $\w^{c}$ of the stable and unstable 
components and thus establish their global smoothness by Journ\'e's Lemma.
 For the center component, in the next section, we will use a different argument 
 based on exponential mixing and  a regularity result from \cite{FKSp13}.
  
   \vskip.1cm
   
Recall that $h\circ f=L\circ h$. We denote by $\bar f$ and $\bar h$ the lifts of $f$ and $h$ to $\R^d$ which are compatible with the standard lift of $L$ so that we have $\bar h\circ \bar f=L\circ \bar h$. Also recall that $h$ is homotopic to the identity
and $f$ is homotopic to $L$. Hence we can write
$$
\bar h=\Id +\bar  H \quad\text{and}\quad  \bar f=L+\bar  F,
$$
where $\bar  H, \bar  F:\R^d \to \R^d$ are $\Z^d$-periodic, and hence can be viewed as functions 
$H$ and $F$ from $\T^d$ to $\R^d$. 
Then the commutation relation 
$$
(\Id + \bar  H)\circ(L + \bar  F)=L\circ (\Id +\bar  H)\quad\text{yields}\quad
\bar  H= L^{-1} (\bar  H\circ \bar f)+L^{-1} \bar  F.
$$
It is easy to check that the latter projects to the torus as the following equation for $\R^d$-valued
functions on $\T^d$
$$
H= L^{-1} (H\circ f)+G, \quad \text{where } G=L^{-1}  F.
$$

Using the $L$-invariant splitting $\R^d= E^u\oplus E^c\oplus E^s$ we define the projections 
$H_*$ and $G_*$ of $H$ and $G$ to $E^*$, where $*=s,u,c$, and obtain  
\begin{equation} \label{H_*}
H_*= L_*^{-1} (H_*\circ f)+G_*, \quad \text{where } L_*=L|_{E^*}.
\end{equation}
Thus $H_*$ is a fixed point of the affine operator 
\begin{equation} \label{T*}
T_* (\psi )=L_*^{-1} (\psi \circ f)+G_*
\end{equation}
with the inverse $T_*^{-1} (\phi )=L_*\, (\phi \circ f^{-1})- L_* \, ( G_* \circ f^{-1})$.

Since $\|L_u^{-1}\|<1$, the operator $T_u$ is a contraction on the space $C^0(\T^d,E^u)$, 
and thus $H_u$ is its unique fixed point 
\begin{equation} \label{Hu}
H_u = \lim_{k\to \infty} T_u^k (0) = \sum_{k=0}^\infty L_u^{-k} (G_u \circ f^k).
\end{equation}
Similarly, $T_s^{-1}$ is a contraction on $C^0(\T^d,E^s)$ and $H_s$ is its 
unique fixed point 
\begin{equation} \label{Hs}
H_s = \lim_{k\to \infty} T_s^{-k} (0) = -\sum_{k=1}^\infty L_s^{k} (G_s \circ f^{-k}).
\end{equation}

\vskip.2cm

Our goal now is to show that $H_c$, $H_u$, and $H_s$ are $C^\infty$, which would yield that $h$ is $C^\infty$. We already know that $h$ is uniformly $C^\infty$ along  $\w^s$ and $\w^u$, 
and hence so are $H_c$, $H_u$, and $H_s$. Thus it remains to study the derivatives 
for each of these maps along $\w^c$.

We will now prove that the derivatives of $H_u$  of any order along $\w^c$ exist and are continuous 
functions on $\T^d$ by term-wise differentiation of \eqref{Hu}, and thus we will show that $H_u$ is uniformly $C^\infty$ along  $\w^c$.

First we observe that the Lyapunov exponents of $\| Df|_{\E^c}\|$ are zero with respect to 
any $f$-invariant measure. Indeed, a non-zero Lyapunov exponent implies exponential 
expansion/contraction by $f$ inside the leaves of $\w^c$, more precisely, existence of $x \in \T^d$
and $y\in \w^c(x)$ such that $\dist(f^nx,f^ny)$ decays exponentially as $n$ goes to $\infty$ or
$-\infty$. Since conjugacy $h$ is H\"older this yields similar exponential decay of $\dist(L^n h(x),L^n h(y))$,
which is impossible as $h(y) \in W^c(h(x))$.  

The fact that the Lyapunov exponents of $\| Df|_{\E^c}\|$ are zero with respect to 
any $f$-invariant measure is well-known to imply that $\| Df^n|_{\E^c}\|$ grows
 sub-exponentially, that is, for any $\e>0$ there is $C_\e$ such that 
\begin{equation} \label{Dfc}
\| Df^n|_{\E^c}\| \le C_\e\, e^{\e n}\quad\text{for all }n \in \N,
\end{equation}
see e.g. \cite{Schr98}.
It follows that the norms of all higher derivatives also grow sub-exponentially, 
 see e.g.  \cite[Lemma 5.5]{LW}: for each $m$ and $\delta>0$ there exists a constant $K_{\delta,\ell}$ such that 
\begin{equation} \label{fcm}
\|f^n \|_{C^\ell_{\w^c}} \le K_{\delta,\ell}\, e^{n\delta} \quad\text{for all }n\in \N,
\end{equation}
where $\| g \|_{C^\ell_{\w^c}}$ denotes the supremum of all derivatives of $g$ of orders up to $\ell$
along the foliation $\w^c$.

 Since $\| L_u^{-1}\|<1$, the above estimate yields that term-wise differentiation of any order 
 of \eqref{Hu} gives  an exponentially converging series. Hence the derivatives of $H_u$ of any 
 order along $\w^c$ are continuous functions on $\T^d$, that is $H_u$ is uniformly $C^\infty$ 
 along $\w^{c}$. We have already established
that $H_u$ is uniformly $C^\infty$ along $\w^u$ and $\w^s$, and so
 we  conclude that $H_u$ is $C^\infty$ on $\T^d$ by Journ\'e's lemma \cite{J}. A similar argument 
 using differentiation of \eqref{Hs} shows that $H_s$ is $C^\infty$ on $\T^d$.
 \vskip.1cm
 
We remark that  term-wise differentiation can be used to establish smoothness of
$H_u$ and $H_c$ along $\w^{s}$ and of $H_s$ and $H_c$ along $\w^{u}$,
but not of $H_u$ along $\w^{u}$ or $H_s$  along $\w^{s}$.


\subsection{Global smoothness of the center component} \label{Csmooth}

In this section we complete the proof of Theorem \ref{localG} by establishing
 global smoothness of $H_c$.  While $H_c$  is a fixed point of the
operator $T_c$ given by \eqref{T*}, $T_c$ is not a contraction  on $C^0(\T^d,E^c)$.
We will show, however, that $H_c$ can be expressed by series similarly to $H_u$ and $H_s$ 
in the sense of distributions
\begin{equation} \label{H_c}
H_c  = \sum_{k=0}^\infty L_c^{-k} (G_c \circ f^k).
\end{equation}
More precisely, we consider the distribution space $\D$ of $E^c$-valued 
functionals $\omega$ on the space of $C^\infty$ test functions 
$\eta : \T^d \to \R$ with zero average with the vector-valued pairing
$$
\langle \omega, \eta  \rangle = \int _{\T^d} \eta (x)\omega (x) \, d\mu(x).
$$
We also fix a norm $|.|$ on $E^c$ to estimate the magnitude.
We need the space of $C^\infty$ test functions only for the formal definition of distributional 
derivatives. All estimates in the proof will be done for a H\"older continuous $\eta$ and all 
distributions will be shown to be defined on the space of H\"older continuous test function.

To verify \eqref{H_c} we  iterate equation 
 \eqref{H_*}, $H_c= L_c^{-1} (H_c\circ f)+G_c$, and get that for any $j\in \N$,
\begin{equation} \label{H_c^m}
H_c = \sum_{k=0}^{j-1} L_c^{-k} (G_c \circ f^k)+ L_c^{-j} (H_c \circ f^j).
\end{equation}
Since $L_c$ is conjugate to an orthogonal matrix, $\| L_c^{-j}\|$ is bounded 
uniformly in $j$. Since $(f,\mu)$ is mixing, as isomorphic to $(L,m)$, we can estimate
 the last term in \eqref{H_c^m}  as  
$$
| \langle L_c^{-j} (H_c \circ f^j),\, \eta  \rangle|= | L_c^{-j} \,\langle  H_c \circ f^j, \eta  \rangle | \le 
\| L_c^{-j} \| \cdot |\langle  H_c \circ f^j, \eta  \rangle | \to 0
$$
as $j\to\infty$\, for any H\"older or  $L^2$ function $\eta$ with 0 average,
 and we conclude that 
\begin{equation} \label{H_c d} 
\langle H_c, \eta  \rangle = \langle \,\sum_{k=0}^{\infty} L_c^{-k} (G_c \circ f^k),\, \eta \,\rangle.
\end{equation}

\vskip.2cm

Now we will prove that $H_c$ is $C^\infty$ on $\T^d$ using a regularity result 
from \cite[Corollary 8.5]{FKSp13},
which yields that it suffices to show that the derivatives of $H_c$ of any order along $\w^c$,
$\w^s$, and $\w^u$ are distributions dual to the space of H\"older functions, i.e., their norms can be estimated
by the H\"older norm of a test function. Recall that the derivatives of $H_c$ of any order along 
$\w^s$ and $\w^u$ are continuous functions by uniform smoothness of $h$ along 
$\w^s$ and $\w^u$ established in Section \ref{hsmoothWs}. 
This can also be seen by term-wise differentiation of the series for $H_c$.
To complete the proof of smoothness of $H_c$, we will now show that the derivatives of $H_c$ 
of any order along $\w^c$ are distributions dual to H\"older functions. 
We use the following result which says that  $L$ has exponential mixing on H\"older functions.
\vskip.15cm
\noindent \cite[Theorem 6]{L}, \cite[Theorem 1.1]{GoSp} 
{\it Let $L$ be an ergodic automorphism of a 
torus, or more generally of a compact nilmanifold $X$. 
Then for any $\theta\in(0,1]$ there exists $\rho=\rho(\theta)\in (0,1)$ such that 
for all $g_0, g_1\in C^\theta(X)$ and $n\in \N$},
$$
  \int_X g_0(x) g_1(L^n(x))\, dm(x)=\left( \int_X g_0 \,dm \right) \left( \int_X g_1 \,dm \right)
  + O(\rho^n \|g_0\|_{C^\theta}\|g_1\|_{C^\theta}).
$$
\vskip.15cm
\noindent Since the bi-H\"older conjugacy $h$ maps the volume $\mu$ to the Lebesgue measure
$m$ and preserves the class of H\"older functions, the same holds for $(f,\mu)$
in place of $(L,m)$.

We fix $\ell\in \N$ and for a smooth function $g$ on $\T^d$ consider its  partial derivative 
$D^{\ell}_\a g$ of order $\ell$ along $\w^c$. We will use the same notation for distributional
derivatives along $\w^c$ (see \cite[Section 8]{FKSp13}
for detailed description of distributional derivatives in the context of foliations).
Using equation \eqref{H_c d} we obtain the formula for distributional derivative of $H_c$,
\begin{equation} \label{h_V^k}
\langle D^{\ell}_\a H_c, \,\eta  \rangle = \langle \,\sum_{k=0}^{\infty} D^{\ell}_\a (L_c^{-k} (G_c \circ f^k)),\, \eta \,\rangle.
\end{equation}
Since $G_c$ and $f$ are smooth, the terms $D^{\ell}_\a (L_c^{-k}  (G_c \circ f^k))$
are continuous functions. Now we estimate these pairings in terms of the H\"{o}lder norm of $\eta$.

We will use smooth approximations of $\eta$ by convolutions with a smooth kernel
$\eta_{\e} = \eta * \phi _{\e}$.  More precisely, we fix a smooth bump 
function $\phi$ supported on the unit ball and define
 $\phi _{\e} (x)= \e ^{-d}  \phi (x /\e)$, so that we have
$$
\phi _{\e} \ge 0, \quad \int _{\T^d} \phi _{\e} =1, \quad
\|  \phi _{\e} \| _{C^\ell}= \e ^{-(d+\ell)}  \|\phi \| _{C^\ell}.
$$
Then  for any $0<\theta \le 1$ and $\ell \in \N$ we have the standard estimates of the norms 
for any $\theta$-H\"{o}lder function $\eta$,
 \begin{equation} \label{fe}
 \begin{aligned}
& \|\eta_{\e } -\eta \|_{C^0} \leq \e ^{\theta} \|\eta\|_{\theta} \;
 \; \text{ and} \;\\
&\|\eta_{\e}\|_{C^\ell}  \leq c_\ell \, \e ^{-d-\ell}  \|\eta\| _{0} \;
 \text{ for } \ell \in \mathbb N,
 \end{aligned}
\end{equation}
\noindent where $c_\ell$ is a constant depending only on $\ell$.
 
 We split $\eta$ as $ \eta_{\e } + (\eta -\eta_{\e })$
 and estimate the corresponding pairings. 
$$
\begin{aligned}
|\langle D^{\ell}_\a L_c^{-k} (G_c \circ f^k),\,\eta_\e\rangle |  &\le
\|L_c^{-k}\| \cdot |\langle D^{\ell}_\a(G_c \circ f^k),\,\eta_\e\rangle | = \\
&= \|L_c^{-k}\| \cdot  |\langle G_c \circ f^k,\,D^{\ell}_\a \eta_\e \rangle |\, .
 \end{aligned}
$$
Since $\| D^{\ell}_\a \eta_\e \|_\theta \le \| D^{\ell}_\a \eta_\e \|_{1}
\le \| \eta_\e \|_{C^{\ell+1}}$,
using the exponential mixing and \eqref{fe} we can estimate
$$
|\langle G_c \circ f^k,\eta_\e^{\ell,c}\rangle |  \le
 K_1 \,  \rho^{k} \, \|G_c \|_\theta \, \| D^{\ell}_\a \eta_\e \|_\theta
  \le K_2 \,  \rho^{k}  \e ^{-(d+\ell+1)} \|G_c \|_\theta \, \|\eta\| _{0} \, .
$$
Since   $\| L_c^{-k}\|$ is bounded we conclude  that
\begin{equation} \label{pfe}
|\langle L_c^{-k} (G_c \circ f^k)^{\ell,c},\,\eta_\e\rangle |  \le
  K_3 \,  \rho^{k}  \e ^{-(d+\ell+1)} \|G_c \|_\theta \, \|\eta\| _{0} \, .\end{equation}

Now we estimate the pairings in \eqref{h_V^k} with $\eta-\eta_\e$. We use an 
estimate on norms of compositions of $C^\ell$ functions
$$\| h\circ g\|_{C^\ell} \le M_\ell \,  \|h\|_{C^\ell} (1+\| g\|_{C^\ell})^\ell ,
$$
which follows, for example, from  Proposition 5.5 in \cite{dlLO}. Thus we have
  $$
  \begin{aligned}
& |\langle D^{\ell}_\a L_c^{-k} (G_c \circ f^k),(\eta-\eta_\e)\rangle | \le
\|L_c^{-k} D^{\ell}_\a (G_c \circ f^k)\|_0 \cdot \| (\eta-\eta_\e)\|_0 \le\\
& \le \|L_c^{-k}\| \cdot \| G_c \circ f^k \|_{C^\ell_{\w^c}} \cdot \e ^{\theta} \|\eta\|_{\theta} \le
K_5 \, \|G_c \|_{C^\ell} \, (1+ \| f^k \|_{C^\ell_{\w^c}})^\ell \cdot\e ^{\theta} \|\eta\|_{\theta}.
\end{aligned}
$$
Now using \eqref{fcm} we obtain 
\begin{equation} \label{pf-fe}
\begin{aligned}
&|\langle L_c^{-k} (G_c \circ f^k)^{\ell,c},(\eta-\eta_\e)\rangle |  
\le K_6\, e^{\ell k\delta}  \cdot \e ^{\theta} \cdot \|  G_c \|_{C^\ell}  \cdot  \|\eta\|_{\theta} = \\
&= K_6\, \xi^k \,\|  G_c \|_{C^\ell}  \cdot  \|\eta\|_{\theta}, \quad\text{where }\, 
\xi=e^{\ell\delta} \e^{\theta / k}.
\end{aligned}
\end{equation}

We choose $\e=\e(k) = \rho^{k/(2(d+\ell+1))}$ so that $\rho^{k} \e ^{-(d+\ell+1)}=\rho^{k/2}$
to obtain exponential decay in \eqref{pfe}.
Then we take $\delta>0$ sufficiently small so that 
$$\xi= e^{\ell\delta}\rho^{\theta/(2(d+\ell+1))}<1,
$$
which ensures exponential decay  in \eqref{pf-fe}.
Noting that $\rho^{1/2} < \xi <1$, we combine \eqref{pfe} and \eqref{pf-fe} to get
$$
|\langle L_c^{-k} (G_c \circ f^k)^{\ell,c},\,\eta\rangle |
\leq  K_7 \, \xi^{k} \cdot \|  G_c \|_{C^\ell} \, \|\eta\|_{\theta}.
$$
Thus, for any $\theta$ and derivative $D^{\ell}_\a $, we obtain exponential convergence 
in \eqref{h_V^k} and conclude that $| \langle D^{\ell}_\a  H_c,\eta \rangle | \le C \|\eta\|_{\theta}$. 
Therefore $D^{\ell}_\a H_c$  extends to a functional on the space of $\theta$-H\"{o}lder functions.
 This concludes the argument that $H_c$ is $C^\infty$ and completes the proof of Theorem \ref{localG}.


\section{Proofs of Corollary \ref{symplecticG}, Theorem \ref{local}, and Corollary \ref{symplectic}}  \label{Pcor}

\subsection{Proof of Corollary \ref{symplecticG}}

We will verify that $\w^c$ is sufficiently smooth, in fact that $\E^c$ is $C^\infty$.
The latter is equivalent to $\E^s\oplus \E^u$ being $C^\infty$ since $\E^c$ is the symplectic 
orthogonal to $\E^s\oplus \E^u$.
Indeed, if $u\in \E^c$ and $v \in \E^s$ then by invariance of the symplectic form $\omega$ 
we have that
$$ 
|\omega_x (v,u)|=|\omega_{f^nx} (D_xf^n(v),D_xf^n (u))| \le C \| D_xf^n(v)\| \cdot \|D_xf^n (u)\| \to 0
$$  
as $n \to \infty$, and so $\omega_x (v,u)=0$. Similarly $\omega_x (v,u)=0$
for any $u\in \E^c$ and $v \in \E^u$.

Now we show that $\E^s\oplus \E^u$ is $C^\infty$.
Since $f$ is topologically conjugate to $L$, the foliations $\w^{u}$ and $\w^{s}$ are topologically 
jointly integrable in the sense that there is a continuous foliation $\w=\w^{s+u}$ of dimension
$\dim \w^s + \dim \w^u$ which is sub-foliated by  $\w^{u}$ and $\w^{s}$.
First we note that the leaves of $\w^{s+u}$ are uniformly $C^\infty$ by the following lemma. 

\begin{lemma} \cite[Lemma 4.1]{KS06} \label{joint int}
Let $\w_1$ and $\w_2$ be foliations with uniformly $C^\infty$ leaves.
Suppose that $\w_1$ and $\w_2$ are topologically jointly integrable to a 
continuous foliation $\w$. Then $\w$ has uniformly $C^\infty$ leaves.
 \end{lemma} 
Now to prove that  $\w^{s+u}$, and hence $\E^s\oplus \E^u$, is $C^\infty$ it suffices
to show that the holonomies of $\w^{s+u}$ between leaves of $\w^c$ are $C^\infty$. 
By the dynamical coherence of $f$, the 
holonomy of $\w^{s+u}$ between the center leaves is smooth as a composition of 
holonomies of $\w^{u}$ inside $\w^{cu}$ and of $\w^{s}$ inside $\w^{cs}$. We claim
that the latter, and similarly the former, holonomies are $C^\infty$, since $\w^{s}$ is a
$C^\infty$ foliation inside the leaves of $\w^{cs}$. For this we note that, as we already observed
in the proof of Theorem \ref{localG}, $Df|_{\E^c}$ has sub-exponential growth, as the exponents 
of $f$ along $\E^c$ are all zero. This implies that $f$ is so called  {\em strongly r-bunched} for
any $r$ and thus the leaves of $\w^{cs}$ are $C^\infty$~\cite{PSW}.  
It also implies that $Df|_{\E^{cs}}$ has sub-exponential growth and so applying the $
C^r$ Section Theorem~\cite{HPS} as for example in
 \cite[Theorem 3.7 and Proposition 3.9]{KS07} we obtain that $\w^{s}$ is
$C^\infty$ along the leaves of $\w^{cs}$.


 \subsection{Proof of Theorem \ref{local}}
 
It is clear that smooth conjugacy implies (1)-(5). Then it suffices to show that all other items
imply (4) and that the topological conjugacy in (4) is bi-H\"older, so that Theorem \ref{localG}
applies and yields smoothness.
 
 The implication (1)$\implies$(2) is clear and the implication (2)$\implies$(3) follows 
 from the next result by Avila and Viana:
 \vskip.15cm
 \noindent \cite[Theorem 8.1]{AV}  
{\it Let $L$ be as in Theorem \ref{local}. Then there exists a neighborhood $U$ 
of $L$ in the space of $C^N$ volume preserving diffeomorphisms of $\T^d$ 
such that if  $f\in U$ is accessible then its center Lyapunov exponents are distinct.}
\vskip.15cm
The (topological) joint integrability of $\w^s\oplus \w^u$ implies that accessibility classes 
are the leaves of $\w^{s+u}$, thus (5)$\implies$(3).
The  implication (3)$\implies$(4) was established in \cite{RH}.
For a perturbation $f$ which is not accessible, it was proved in  \cite[Section 6]{RH} 
(cf. \cite[Remark 8.3]{AV}) that $f$ and $L$ are conjugate by a bi-H\"older  homeomorphism $h$. 
This completes the proof of Theorem \ref{local}.


\subsection{Proof of Corollary \ref{symplectic}}

Combining the above proof of Theorem \ref{local} with Corollary \ref{symplecticG}
we conclude that smooth conjugacy in this case is equivalent to (1)-(5). 
A smooth conjugacy also clearly implies (0) and (6). 

For a symplectic $f$ the Lyapunov spectrum is a symmetric subset of $\R$, that is, 
the Lyapunov exponents come in pairs $\lambda, -\lambda$.
Indeed, let $\omega$ be the invariant symplectic form. Since $\omega$ is non-degenerate, 
each Lyapunov space $\E^i$ is not symplectic orthogonal to at least one Lyapunov space 
$\E^j$. Then for suitable vectors $v_i$ and $v_j $  in these spaces we have by invariance that
 $$
0\ne \omega_x (v_i,v_j)=
\omega_{f^nx} (D_xf^n(v_i),\,D_xf^n (v_j)).$$
This implies $ \lambda_i^f+\lambda_j^f=0$ as otherwise the right hand side must go to 0 under 
forward or backward iterates.

Since $\E^c$ is symplectic orthogonal to $\E^s\oplus \E^u$, the argument above also shows that
the center exponents are of the form $\lambda, -\lambda$, and thus (0)$\implies$(1).
 
Finally, (6) implies (5) or (0) by a result of 
Hammerlindl \cite[Theorem 1.1]{H}: if $\E^s\oplus \E^u$ is $C^1$ and 
not integrable, then a center exponent must be qual
to the sum of a stable one and an unstable one. We let
$$
2\varepsilon=  \min_{\lambda_i\neq\pm\lambda_j}|\,|\lambda_i^L|-|\lambda_j^L|\,|>0.
$$
If $f$ is  sufficiently $C^1$ close to $L$ then the similar minimum for $f$ is at least $\varepsilon$
while the center exponents satisfy $|\lambda_c^f|<\varepsilon$. Then, by the symmetry of
Lyapunov spectrum, the equation $\lambda_c^f=\lambda_i^f+\lambda_j^f$
can only hold in the case $\lambda_i+\lambda_j=0$, yielding  (0).

This completes the proof of Corollary \ref{symplectic}.


\section{Proof of Theorem \ref{localL}} \label{PlocalL}

\subsection{Outline of the proof}

The main part of the proof is establishing smoothness of the leaf conjugacy transversely 
to the center foliation.  This is similar in spirit to proving smoothness of the conjugacy
along the stable and unstable foliations in Section~\ref{hsmoothWs}. However, in absence
of a true conjugacy, the argument with holonomies and 
normal forms  becomes more difficult. 

  As before, we denote the stable, unstable, and center sub-bundles for $L$
 by $E^s$, $E^u$, $E^c$, and the ones for $f$ by $\E^s$, $\E^u$, $\E^c$.
 Similarly, we use $W$ and $\w$ for the corresponding foliations for $L$ 
 and  for $f$.
 
 We recall that there exists a leaf conjugacy $h$,
that is, a homeomorphism close to the identity which maps center leaves to center leaves 
and conjugates $f$ to $L$ modulo the center foliation~\cite{HPS}. Further, it maps center-stable leaves to center-stables leaves and center-unstable leaves to center-unstable leaves. Such a leaf conjugacy can always be chosen to be smooth along the center foliation $\w^c$. In fact, this is true in general, for partially hyperbolic diffeomorphisms with $C^1$ center foliation. We can make a specific choice of $h$ as follows. Denote by $\bar h\colon\R^d\to\R^d$ the lift of $h$ that we would like to define. On the universal cover we have a direct splitting $\R^d=E^s\oplus E^c\oplus E^u$ and hence we can use $(s,c,u)$-coordinates
$$
\bar h(x)=\bar h(x_s, x_c, x_u)=(\bar h_s(x), \bar h_c(x), \bar h_u(x)).
$$
In the notations $\bar h(x)=x+\bar H(x)$ of Section~\ref{Gsmooth} this corresponds to $\bar h_*(x)=x_*+H_*(x)$.
Then $\bar h_s$ and $\bar h_u$ are uniquely determined by \eqref{Hu} and \eqref{Hs}. We take $\bar h_c(x)=x_c$, which corresponds to setting $\bar H^c=0$. Then $\bar  h_c$ is obviously smooth. Because $\bar h$ sends center leaves to center leaves, if $x$ varies in $\w^c(x_0)$ then $\bar h(x)$ varies in $W^c(\bar h(x_0))$ and, hence, the coordinates $\bar h_s$ and $\bar h_u$ do not change. In the same way, if $x$ varies in $\w^s(x_0)$ then $\bar h(x)$ varies in $W^{cs}(\bar h(x_0))$ and, hence, the coordinate $\bar h_u$ does not change. And when $x$ varies in $\w^u(x_0)$ the coordinate $\bar h_s$ does not change.
Hence, to prove that $h$ is $C^\infty$ it suffices to show that $\bar h_s$ is uniformly $C^\infty$ along $\w^s$. 
This is done in Section \ref{hononomyL} below. Similarly, $\bar h_u$ is uniformly $C^\infty$ along $\w^u$, 
which completes the proof.


\subsection{Smoothness of $h_s$ along $\w^s$}  \label{hononomyL}
In this section we give modifications needed to carry out the arguments from Section \ref{hsmoothWs}
 in the case of leaf conjugacy. The main part is to establish the following generalization of Proposition \ref{normal holonomy}.
  
  \begin{proposition}  \label{normal holonomy L} 
 Let $L: \T^d\to\T^d$ be a partially hyperbolic automorphism which is diagonalizable over $\C$.
  Let $f\colon\T^d\to \T^d$ be a sufficiently $C^1$-small perturbation of $L$.
  Let  $\{\h_x\}_{x\in \T^d}$ be normal form coordinates for $f$ on $\w^s$, as in Theorem~\ref{MainNF}.
 For any $x\in \T^d$ and $y\in \w^c(x)$ the center holonomy $\H_{x,y}:\w^{s}(x) \to \w^{s}(y)$ 
 preserves normal forms, that is, the map
 $$
  \h_y \circ  \H_{x,y} \circ  \h_x^{-1} :\E_{x}^s \to \E_{y}^s \;\text{ is in  }\,\p _{L_s},
  $$ 
  the group of sub-resonance generated polynomial map 
 defined by $L_s=L|_{E^s}$.
 \end{proposition}

\begin{proof}
As in the proof of  Proposition \ref{normal holonomy}, we consider the mapping tori
and the corresponding suspension flows $f^t$ and $L^t$. Then the leaf conjugacy $h$, 
which was chosen in the previous subsection, induces 
the leaf conjugacy $\th\colon M_f\to M_L$  between  the suspension flows given by $\th (x,t)=(h(x), t)$.

We recall that $L$ is diagonalizable and all its eigenvalues on $E^c$ have modulus 1. 
Hence we can decompose $E^c=\oplus V_j$ as the direct sum of eigenspaces corresponding
to eigenvalues $1$ and $-1$ and of $L$-invariant subspaces corresponding to pairs of complex 
eigenvalues $e^{\pm 2\pi i \theta_j}$. We will consider center holonomies corresponding to each
of these subspaces separately. We fix one of the subspaces corresponding to a complex pair
and write $V=V_j$ and $\theta=\theta_j$. The case when $V$ is an eigenspace of $1$ or $-1$ can be considered similarly using $L^2$ in place of $L^{1/\theta}$.

For any $v\in V$ we again consider the translation $H_v (x) =x+v$, $x\in\T^d$, which 
embeds as $t=0$\, level of  the map $\tilde H_v: M_L\to M_L$. The map
$\tilde H_v$ commutes with $L^{1/\theta}$, the time $1/\theta$ map  of the suspension flow. 
Since $h$ is not a conjugacy, we will first consider the normal forms for a different dynamics on $M_f$. 
Denoting 
$$
\f^t= \tilde h^{-1} \circ L^{t}\circ \tilde h \quad\text{for }t\in \R,
$$
 we obtain a continuous 
flow on $M_f$. We fix $v\in E^c$ and define  the homeomorphism 
$$
g=g_v=\tilde h^{-1} \circ \tilde H_v\circ \tilde h
$$  
 which again  commutes with $\f^{1/\theta}$.  
However, $\f^t$  and $g$ may not preserve the foliation $\tilde \w^s$.
Since $h$ is a leaf conjugacy, the homeomorphisms $\f^t$  and $g$ preserve foliations 
$\tilde \w^c$ and $\tilde \w^{cs}$, and they differ from $f^{t}$ and from 
a center holonomy between strong leaves respectively,  by ``adjusting along the center".
More precisely, 
$$
\f^t(x) \in \tilde \w^c (f^{t}(x))\quad\text{and}\quad g(x)\in \tilde \w^c (x).
$$

Now we define smooth extensions $F^t$ and $G$ of $\f^t$ and $g$. 
They reflect the behavior of $f^{t}$ and of the center holonomies between the 
corresponding strong stable leaves. We fix $x\in M_f$ and for each $t \in \R$ we define 
$$
F^{t}_x: \tilde \w^s(x) \to \tilde \w^s(\f(x))\;\text{ as }\;
F^{t}_x = \P_{\f(x)} \circ \f^{t} |_{\tilde \w^s(x)}, 
$$
where $\P_{\f(x)}$ is 
``holonomy projection" along  $\tilde \w^{c}$ inside the leaf of $\tilde \w^{cs}$, that is 
 $$
 \P_{x} =\P^{cs}_{x} :\tilde \w^{cs}(x) \to \tilde \w^{s}(x) \quad\text{given by}\quad
  \P_{x}(z)= \tilde \w^c(z)\cap  \tilde \w^{s}(x).
  $$
Note that $\P_{x}$ is globally defined on $\tilde \w^{cs}$ since the leaf conjugacy
$h$ maps the leaves of $\tilde \w^{c}$ and $\tilde \w^{cs}$ to those of $\tilde W^{c}$ 
and $\tilde W^{cs}$. Also, since $\f^{t}(y)$ and $f^{t}(y)$ are on the same leaf of $\tilde \w^c$,
we can also express $F_x^t$ as 
\begin{equation} \label{F^t}
 F^{t}_x = \H_{f^{t}(x),\f^t(x)} \circ f^{t} |_{\tilde \w^s(x)}: \tilde  \w^{s}(x) \to \tilde \w^{s}(\f(x)),
\end{equation}
 where $\H_{x,y}: \tilde \w^s(x) \to \tilde \w^s(y)$ denotes the usual $\tilde \w^c$ holonomy.
Similarly, for any $x\in M_f$, we define 
$$
G_x: \tilde \w^s(x) \to \tilde \w^s(g(x))\quad\text{as} \quad
G_x = \P_{g(x)} \circ g|_{\tilde \w^s(x)}.
$$
Since $\f^{1/\theta}$ and $g$ commute, it is clear from 
 the definitions that the extensions also commute: 
 $G_{\f (x)} \circ F^{1/\theta}_x=  F_{g(x)}^{1/\theta} \circ G_x$. 
 Again, as  $g(y)\in \tilde \w^c (y)$ we see that $G_x$ coincides with the center holonomy 
 \begin{equation} \label{G}
 G_x=\H_{x,g(x)} :\tilde \w^{s}(x) \to \tilde \w^{s}(g(x)).
\end{equation}
 Since $\tilde \w^c$ is a $C^\infty$ foliation, $\P_{\f(x)}$ and $\H_{x,g(x)}$ are $C^\infty$, and 
 thus both $F_x^t$ and $G_x$ are $C^\infty$ diffeomorphisms.
 
  Now we construct normal forms for the extension $F^t$ and show that $G$ preserves them.
 To apply Theorem \ref{NFext}, we locally identify $\tilde \w^{s}(x)$ and $\tilde \E^s_x$ and
 obtain the corresponding smooth  extensions $\bar F^t$ and $\bar G$ of
 $\f^t$ and $g$, respectively, defined on a neighborhood of the zero section in $\E=\tilde \E^s$.
We claim that the derivative of $\bar F^t$ at the zero section is a contraction which is close to
the linear flow $L^t$, provided that $f$ is sufficiently $C^1$ close to $L$.  Indeed,
differentiating \eqref{F^t} at $x$ we obtain 
$$
D_0 \bar F^t_x=D_x F^t_x = D_{f^{t}(x)} \H_{f^{t}(x),\f^t(x)} \circ Df^{t} |_{\tilde \E^s(x)}.
$$
If $f$ is sufficiently $C^1$ close to $L$ then $h$ is $C^0$ close to the identity, and hence
$\f^t$ is $C^0$ close to $f^{t}$. Thus $\f^t (x)$ is close to $f^{t}(x)$ 
and hence the derivative of the holonomy $\H_{f^{t}(x),\f^t(x)} $
is close to the identity. Thus $D_0 \bar F^t_x$ is close to $D f^{t} |_{\tilde \E^s(x)}$, 
which is close to $L^t$.
In particular, $D_0 \bar F^1$  is close to $L$ and, as $\bar F^t$ is $C^\infty$, we can now 
apply Theorem \ref{NFext} with $F=F^1$ and $A=L_s=L |_{E^s}$ 
to obtain a family of local normal form coordinates 
$\bar \h_x$ for $\bar F^1$ on $\tilde \E^s$. 

Since all maps $\bar F^t$ commute, the second part of 
Theorem \ref{NFext} implies that $\bar \h_x$ are also normal form coordinates for the whole 
one-parameter group $\{\bar F^t\}$. 
Hence by the identification we obtain local normal form coordinates $\h_x$ for $F^t$ on $\tilde \w^s$. 
Then we can extend $\h_x$, as in the Remark \ref{NFglob}, to get global normal form coordinates 
on the whole leaves $\h_x :  \tilde \w_{x}^s \to \tilde \E_{x}^s$.  
Indeed, while $F^t_x$ may not be a global contraction, 
for any bounded set $B\subset  \tilde \w_{x}^s$, the set $F^n_x(B)$ will be in a small neighborhood 
of $\f^n(x)$ for all sufficiently large $n$, and hence we can define $\h_x $ on $B$  by
$$
\h_x  = (\p_x^n)^{-1} \circ  \h_{\f ^n (x)} \circ F_x^n.
$$

Since the extension $G$ is also $C^\infty$ and commutes with $ F^{1/\theta}$, the second part 
of Theorem \ref{NFext} implies that $G$ preserves the normal form coordinates
 for $F^t$ on $\tilde \w^s$, i.e., $\h_{g(x)} \circ G_x \circ \h_x^{-1} \in \p_{L_s}$,
  the sub-resonance group given by $A={L_s}$.
By \eqref{G}, $G_x$ is the holonomy $\H_{x,g(x)} :\tilde \w^{s}(x) \to \tilde \w^{s}(g(x))$ and
we conclude that 
$$
\h_{g(x)} \circ \H_{x,g(x)} \circ \h_x^{-1} \in \p_{L_s}.
$$

Recall that $E^c=\oplus V_j$. The above conclusion holds for $g_v=h^{-1}\circ \tilde H_v\circ h$, 
where $v$ is any vector in any $V_j$. We decompose any vector $w\in E^c$ as the sum 
$w=\sum v_j$ and note that the holonomy $\tilde H_w$ is the composition of the holonomies
$\tilde H_{v_j}$. Therefore, $g_w=h^{-1}\circ \tilde H_w\circ h$ preserves normal forms
as the corresponding composition of the maps $g_{v_j}$. Since for any
$x \in M_f$ and any $y \in \w^c(x)$ we can take $w=h(y)-h(x)$ so that $g_w(x)=y$, we conclude that any 
center holonomy map $\H_{x,y} :\tilde \w^{s}(x) \to  \tilde  \w^{s}(y)$ preserves normal forms.

Considering $t=0$ level of the suspension $M_f$ we obtain this result for $\T^d$:
for any $x \in \T^d$ and any $y \in \w^c(x)$
$$ 
\h_y \circ  \H_{x,y} \circ  \h_x^{-1} :\E_{x}^s \to \E_{y}^s \quad \text{is in} \; \p_{L_s}.
$$ 
 
Finally, we note that by \eqref{F^t} we have $ F^1_x = \H_{f(x),\f(x)} \circ f |_{\w^s(x)}$.
Since both $F^1_x$ and the holonomy are in $\p_{L_s}$, we conclude that so is $ f|_{\w^s(x)}$.
Therefore,  $\{\h_x\}_{x\in \T^d}$ are normal form coordinates for $f$ on $\w^s$,
as in Theorem~\ref{MainNF}. This completes the proof of Proposition \ref{normal holonomy L}. 
\end{proof}

Now we show that $h_s$ is uniformly $C^\infty$ along $\w^s$.
We fix a point $x\in \T^d$ and consider the map 
$$
\hat h_x :\w^s(x) \to W^s(h(x))\quad\text{given by}\quad 
\hat h_x = H^c_{h(x)} \circ h|_{\w^s(x)},
$$
 where $H^c_{h(x)}$ is the linear projection inside
$W^{cs}(h(x))$ to $W^s(h(x))$ along $W^c$. We will prove that $\hat h_x$ is uniformly $C^\infty$.
This will show that the component $h_s$ is uniformly $C^\infty$ along the leaves of  $\w^s$, 
as it is easy to see that $\hat h_x = h_s |_{\w^s(x)}$ under the natural identification of $W^s(h(x))$ with $E^c$.

We fix $y\in \w^{s}(x)$ and take a sequence of points $y_n\in \w^c(x)$ converging to $y$.
This can be done since the leaves of the linear foliation $W^c$ are dense in $\T^d$ and 
the fact that the leaf conjugacy $h$ is a homeomorphism which sends $\w^c$ to $W^c$.
We consider holonomies $\H_{x,y_n}: \w^s(x) \to \w^s(y_n)$ of $\w^{c}$ inside $\w^{cs}$. 
We claim that the holonomy maps $\H_{x,y_n}$ converge in $C^0$ to the map 
$\H_{x,y}: \w^s(x) \to \w^s(y)$,  which is conjugate by $\hat h_x$ to linear translation 
$H_{\hat v}$ in $W^s(h(x))$ by the vector 
$$
\hat v=\hat h_x(y)-\hat h_x (x)=\hat h_x(y)- h (x).
$$ 
Indeed, since $y_n\in \w^c(x)$ converge to $y$,  $h(y_n) \in \w^c (h(x))$ converge to $h(y)$.
The corresponding linear center holonomies $H_{h(x),\,h(y_n)}$ for $L$ are translations
$H_{v_n}$ by the vectors $v_n=h(y_n)-h(x)$ and thus converge in $C^0$ to the 
translation $H_v : W^s(h(x)) \to W^s(h(y))$ by the vector $v=h(y)-h(x)$. Composing with
the translation $H_{\hat v -v}$, which is also a linear center holonomy, we see that
$$
H_{\hat v -v} \circ H_{h(x),\,h(y_n)} \;\text{ converges to }\;
H_{\hat v}: W^s(h(x)) \to W^s(h(x))=W^s(\hat h_x(y)),
$$ 
and that the map
$$
(\hat h_x)^{-1}\circ  H_{\hat v}\circ  \hat h_x : \w^s(x) \to \w^s(x)=\w^s(y)
$$ 
is the limit $\H_{x,y}$ of the holonomies $\H_{x,y_n}: \w^s(x) \to \w^s(y_n)$.
\vskip.1cm

Once we have this convergence of $\H_{x,y_n}$ to $\H_{x,y}$ and 
Proposition \ref{normal holonomy L}, we can use the same normal form argument 
as in Section \ref{hsmoothWs}. Indeed, we again obtain that 
$P_n =  \h_{y_n} \circ  \H_{x,y_n} \circ  \h_x^{-1}$
and their $C^0$ limit $P=  \h_{y} \circ  \H_{x,y} \circ  \h_x^{-1}$ are sub-resonance 
generated polynomials.  Identifying $\w^s(x)$ with $\E^s_x$ by $\h_x$ we obtain that
$P$ is in the Lie group $\bar \p _x$ generated by the translations of $\E^s_x$ and 
the sub-resonance generated polynomials. Then $\hat h$ defines the continuous 
homomorphism 
 $$
 \eta_x : E^s \to  \bar \p_x \quad\text{given by}\quad 
 \eta_x (\hat v)= (\hat h_x)^{-1}\circ  H_{\hat v}\circ  \hat h_x,
 $$ 
 which are $C^\infty$. 
 This yields  that $\hat h^{-1}_x$ and $\hat h_x$ are $C^\infty$ diffeomorphisms that
  depend continuously on $x$ in $C^\infty$ topology. 

   This shows that  $h_s$ is uniformly $C^\infty$ along $\w^s$ and completes
  the proof of Theorem~\ref{localL}.




\begin{thebibliography}{XXXXX}

\bibitem[AV10]{AV} A. Avila, M. Viana. {\em Extremal Lyapunov exponents: an invariance 
            principle and applications.} Inventiones Mathematicae  (2010) Vol. 181, Issue 1, 115-178.
            
\bibitem[AVW15]{AVW1} A. Avila, M. Viana, A. Wilkinson. {\em Absolute continuity, Lyapunov exponents and rigidity I: geodesic flows.} J. Eur. Math. Soc. (JEMS) 17 (2015), no. 6, 1435-1462. 
         
\bibitem[AVW]{AVW2} A. Avila, M. Viana, A. Wilkinson. 
{\em Absolute continuity, Lyapunov exponents and rigidity II.} Preprint.


\bibitem[B67]{B} K. Berg. {\it Entropy of torus automorphisms.} 1968 Topological Dynamics (Symposium, Colorado State Univ., Ft. Collins, Colo., 1967)  67-79, Benjamin, New York.

\bibitem[DX17]{DX} D. Damjanovic, D. Xu. {\em On conservative partially hyperbolic abelian actions with compact center leaves. } Preprint.


\bibitem[dlLO98]{dlLO}  R. de la Llave and R. Obaya.
{\em Regularity of the composition operator in spaces of H\"older functions.}
Discrete and Continuous Dynamical Systems. 5 (1999), no. 1, 157-184.

\bibitem[dlLW10]{LW}  R. de la Llave and A. Windsor. 
             {\em Liv\v{s}ic theorem for non-commutative groups including groups  
              of diffeomorphisms, and invariant geometric structures.} 
               Ergodic  Theory   Dynam. Systems, 30, no. 4  (2010),  1055-1100.               

\bibitem[FKSp11]{FKSp11} D. Fisher, B. Kalinin, R. Spatzier. {\em Totally non-symplectic Anosov actions on tori and nilmanifolds.} Geometry and Topology {15} (2011) 191-216.

\bibitem[FKSp13]{FKSp13} D. Fisher, B. Kalinin, R. Spatzier. {\em Global rigidity of higher rank Anosov actions on tori and nilmanifolds}. 
 J. Amer. Math. Soc., { 26} (2013), no. 1, 167-198.

\bibitem[GoSp14]{GoSp} A. Gorodnik, R. Spatzier. {\em Exponential mixing of nilmanifold automorphisms.}
 Journal d'Analyse, 123 (2014), 355-396.

\bibitem[GKS18]{GKS2}    A. Gogolev,  B. Kalinin, V. Sadovskaya.\, 
{\em Local rigidity of Lyapunov spectrum for toral automorphisms.} To appear in Israel J. Math.

\bibitem[Gu02]{G} M. Guysinsky. {\em The theory of non-stationary normal forms.}
 Ergodic Theory  Dynamical  Systems, 22 (3), (2002), 845-862.
 
\bibitem[GuKt98]{GK}  M. Guysinsky,  A. Katok. {\em Normal forms and invariant 
              geometric structures for dynamical systems with invariant  contracting foliations.} 
              Math. Research Letters {\bf 5}  (1998), 149-163.

\bibitem[Ha]{Ha} B. Hall.  {\em Lie Groups, Lie Algebras, and Representations: An Elementary Introduction.} Graduate Texts in Mathematics, 222, 2nd ed., (2015) Springer. 
              
\bibitem[H11]{H} A. Hammerlindl. {\it Integrability and Lyapunov exponents.}
   J. Modern Dynamics, 5 (2011), no. 1, 107-122.

 \bibitem[HPS77]{HPS} M. Hirsch, C. Pugh, M. Shub.
        {\em Invariant manifolds.}
        Lecture Notes in Math., 583, Springer-Verlag, (1977).
       
\bibitem[J88]{J}   J.-L. Journ\'e.
          {\em A regularity lemma for functions of several variables.}
          Revista Matem\'atica Iberoamericana {4} (1988), no. 2, 187-193.
          
\bibitem[K19]{K19} B. Kalinin.   
       {\em Non-stationary normal forms for contracting extensions}. Preprint.
  
\bibitem[KS06]{KS06}  B. Kalinin, V. Sadovskaya.  
         {\em Global Rigidity for TNS Anosov  $\Z^k$ Actions.}   
         Geometry and Topology, {\bf 10} (2006), 929-954.

\bibitem[KS07]{KS07} B. Kalinin, V. Sadovskaya.    
   {\em On classification of resonance-free Anosov $\Z^k$ actions.}         
   Michigan Math. Journal, {55} (2007), no. 3, 651-670.

\bibitem[KS16]{KS16} B. Kalinin, V. Sadovskaya.   
{\em Normal forms on contracting foliations: smoothness and homogeneous  structure.} 
Geometriae Dedicata, Vol. 183 (2016), no. 1, 181-194.

 \bibitem[KS17]{KS17} B. Kalinin, V. Sadovskaya.  
  {\em Normal forms for non-uniform contractions}.  
Journal of Modern Dynamics, vol. 11 (2017), 341-368.  

                               
\bibitem[KtL91]{KL}  A. Katok, J. Lewis. {\em Local rigidity for certain groups of 
              toral automorphisms.} Israel J. Math. {\bf 75} (1991), 203--241.
  
 \bibitem[KtSp97]{KtSp} A. Katok, R. Spatzier. {\it Differential rigidity of Anosov actions of higher rank abelian groups and algebraic lattice actions. } Tr. Mat. Inst. Steklova 216 (1997), Din. Sist. i Smezhnye Vopr., 292--319; translation in Proc. Steklov Inst. Math. 1997, no. 1 (216), 287-314.
  
  \bibitem[Kz71]{Kz}  Y. Katznelson. {\em Ergodic automorphisms of $\T^n$ are Bernoulli shifts.}
   Israel J. Math. 10 (1971), 186-195.
            
 \bibitem[L82]{L}    D. Lind. {\em Dynamical properties of quasihyperbolic toral automorphisms.}
               Ergodic Theory Dynamical Systems, 2 (1982), no. 1, 49Ð68.
              
\bibitem[M80]{M} B. Marcus. {\em A note on periodic points of toral automorphisms.}
 Monatsh. Math. 89 (1980), 121-129.

\bibitem[MZ74]{MZ}  D. Montgomery, L. Zippin. {\em Topological transformation groups.}  
Robert E. Krieger Publishing Co., Huntington, N.Y., (1974). 
MR0379739 Reprint of the 1955 original.

\bibitem[O68]{O} V. Oseledets. {\em A multiplicative ergodic theorem. Liapunov characteristic 
numbers for dynamical systems.} Trans. Mosc. Math. Soc. 19 (1968), 197-221.

             
\bibitem[PSW97]{PSW} C. Pugh, M. Shub, A. Wilkinson. {\it H\"older foliations.} 
       Duke Math. J. Volume 86, Number 3 (1997), 517-546.
             

\bibitem[RH05]{RH}  F. Rodriguez Hertz. {\em Stable ergodicity of certain linear 
             automorphisms of the torus.} Annals of Math., 162 (2005), 65-107.              

\bibitem[Schr98]{Schr98} S. J. Schreiber. 
            {\em On growth rates of subadditive functions for semi-flows.}
              J. Differential Equations, 148 (1998), 334-350.
              
\bibitem[SY18]{SY} R. Saghin, J. Yang. 
      {\em Lyapunov exponents and rigidity of Anosov automorphisms and skew products.} 
      Preprint.




\bibitem[V86]{V} W. Veech. 
{\em Periodic points and invariant pseudomeasures for toral endomorphisms.}\\
 Ergodic Theory  Dynamical  Systems, 6 (1986), 449-473.
 
\end{thebibliography}
\end{document}